\documentclass{amsart} 
\usepackage{amssymb,amsmath,latexsym,times}
\usepackage[dvips]{graphicx}

\newtheorem{thm}{Theorem}[section]
\newtheorem{dfn}[thm]{Definition}
\newtheorem{cor}[thm]{Corollary}
\newtheorem{prop}[thm]{Proposition}
\newtheorem{lemma}[thm]{Lemma}

\newcommand{\del}{\backslash}
\newcommand{\cl}{\hbox{\rm cl}}

\newfont{\menutt}{cmtt8}

\title[Lattice Path Matroids: The Excluded Minors]{Lattice Path Matroids:
  The Excluded Minors} \date{\today}
\author
       {Joseph E.~Bonin}
\address
{Department of Mathematics\\ The George Washington University\\
Washington, D.C. 20052} \email 
{jbonin@gwu.edu}

\begin{document}

\begin{abstract}
  A lattice path matroid is a transversal matroid for which some
  collection of incomparable intervals in some linear order on the
  ground set is a presentation.  We characterize the minor-closed
  class of lattice path matroids by its excluded minors.
\end{abstract}

\maketitle

\section{Introduction}

Among transversal matroids, it is natural to consider those for which
some presentations have special structure.  We consider transversal
matroids for which at least one presentation consists of intervals in
some linear order on the ground set and no interval contains another;
this gives the class $\mathcal{L}$ of lattice path matroids.  Unlike
the class of all transversal matroids, if a matroid is in
$\mathcal{L}$, then so are its minors.  A major theme in matroid
theory is characterizing minor-closed classes of matroids by their
excluded minors, that is, by the minor-minimal matroids that are not
in the class; we give such a characterization of $\mathcal{L}$.  After
a section of background, this result and its proof occupy the rest of
the paper.

We briefly sketch this research area.  Nested matroids, the
minor-closed subclass of $\mathcal{L}$ whose members have
presentations that are chains of intervals in linear orders, have been
introduced many times and under many names (see \cite{lpm2}),
apparently first by H.~Crapo~\cite{single}.  As N.~White noted in his
review of~\cite{lpm2} in Mathematical Reviews, R.~Stanley mentioned
lattice path matroids (without this name) in~\cite{stanley}; no
results were given.  Independently, J.~Lawrence~\cite{jim} introduced
and studied oriented counterparts of these matroids.  Lattice path
matroids were independently introduced and studied in depth
in~\cite{lpm1,lpm2}; the lattice path perspective used there
accounts for the name.  They have been studied further by
J.~Schweig~\cite{jay} and applied to a problem in enumeration by A.~de
Mier and M.~Noy~\cite{tennis}.  A larger minor-closed class of
transversal matroids was defined and studied in~\cite{omer}.
In~\cite{flag}, A.~de Mier used lattice paths in higher dimensions to
define a related type of flag matroid.  Following a suggestion by
V.~Reiner~\cite{vic} that there should be a type-B counterpart of the
Catalan matroid (a certain nested matroid), J.~Bonin and A.~de Mier
defined a class of Lagrangian matroids based on lattice paths; this
topic has been studied by A.~Gundert, E.~Kim, and
D.~Schymura~\cite{crm}.

\section{Background}\label{sec:background}

We assume readers know basic matroid theory; see~\cite{ox} for an
excellent account.  The results we use to prove the excluded-minor
characterization of $\mathcal{L}$ are collected below.

\subsection{Connected and cyclic flats}
A flat $F$ of a matroid $M$ is \emph{connected} if the restriction
$M|F$ is connected.  Only connected flats of rank zero or one can be
trivial, that is, independent.  The following lemma, akin
to~\cite[Exercise 2.1.13]{ox}, is easy to prove.

\begin{lemma}
  A loopless matroid is determined by its collection of nontrivial
  connected flats and their ranks.
\end{lemma}

\begin{lemma}\label{lem:stayconn}
  If $F$ is a flat of $M$ and $x\in E(M)-F$, then $M/x|F = M|F$; thus,
  if $F$ is connected and $\cl_M(x)=\{x\}$, then $\cl_{M/x}(F)$ is a
  connected flat of $M/x$.
\end{lemma}

The flats in the following definition play major roles in our work.
(All references to incomparability in this paper are with respect to
containment.)

\begin{dfn} 
  The proper nontrivial connected flats of a matroid are its
  \emph{pnc-flats}.  A pnc-flat is \emph{reducible} if it is the
  intersection of some pair of incomparable pnc-flats.  A
  \emph{fundamental flat} is a pnc-flat $F$ such that, for some
  spanning circuit $C$ of the matroid, $F\cap C$ is a basis of $F$.
\end{dfn}

A set in a matroid is \emph{cyclic} if it is a (possibly empty) union
of circuits.  Apart from singleton flats of rank one, connected flats
are cyclic.  Cyclic flats may be disconnected.  Note that if $F$ and
$G$ are cyclic flats of $M$ with $F\subsetneq G$, then $|G-F|\geq 2$.
The following lemma~\cite[Exercise 2.1.13]{ox} relates the cyclic
flats of a matroid and those of its dual.

\begin{lemma}
  A set $F\subseteq E(M)$ is a cyclic flat of $M$ if and only if
  $E(M)-F$ is a cyclic flat of $M^*$.
\end{lemma}

\subsection{Lattice path matroids}\label{sec:lpms}
While many results in this subsection are from~\cite{lpm2}, some are
extensions or refinements that are tailored to the work in this paper.

A \emph{lattice path matroid} is a transversal matroid that has a
presentation by an antichain of intervals in some linear order on the
ground set.  (It is easy to translate between this and the lattice
path view in~\cite{lpm2}.)  Thus, such a transversal matroid $M$ of
rank $r$ has a presentation $\mathcal{A}=(J_1,J_2,\ldots,J_r)$ where,
relative to some linear order $e_1<e_2<\cdots<e_n$ on $E(M)$, we have
elements $a_1<a_2<\cdots<a_r$ and $b_1<b_2<\cdots<b_r$ with $a_i\leq
b_i$ and $J_i=[a_i,b_i]$.  Such a linear order is a \emph{path order}
of $M$, the elements $e_1$ and $e_n$ are \emph{terminal elements} of
$M$, and we call $\mathcal{A}$ an \emph{interval presentation} of $M$.
(The term path order reflects the fact that with such an order, it is
easy to give a bijection between the bases of $M$ and the lattice
paths in a certain region of the plane.)

Let $\mathcal{L}$ be the class of lattice path matroids.  A matroid in
$\mathcal{L}$ may have many path orders and many terminal elements;
for example, all uniform matroids are in $\mathcal{L}$ and all their
linear orders are path orders.  Note that if $e_1<e_2<\cdots<e_n$ is a
path order of $M\in \mathcal{L}$, then so is $e_1>e_2>\cdots>e_n$.
The following result recasts~\cite[Theorem~5.6]{lpm2}.

\begin{prop}\label{prop:uniqpath}
  For each path order of $M\in \mathcal{L}$, there is only one
  interval presentation.
\end{prop}

The class $\mathcal{L}$ is easily seen to be closed under direct sums.
Connectivity can be determined readily from any interval
presentation~\cite[Theorem~3.5]{lpm2}.

\begin{prop}
  Given an interval presentation $\mathcal{A}$ of $M\in \mathcal{L}$
  as above, $M$ is connected if and only if $a_1=e_1$, $b_r=e_n$, and
  $a_{i+1}\leq b_i$ for all $i$ with $1\leq i<r$.
\end{prop}

The class $\mathcal{L}$ is closed under minors~\cite[Theorem
3.1]{lpm2}.  Furthermore, given a path order of $M\in \mathcal{L}$ and
a minor $N$ of $M$, the induced order on $E(N)$ is a path order.
Thus, if $x\in E(N)$ is a terminal element of $M$, then $x$ is a
terminal element of $N$.

The next proposition gives the interval presentation of a
single-element contraction.  We first note that if a presentation
$\mathcal{A}$ is as given above, then the sets in $\mathcal{A}$ that
contain a particular non-loop $y$ are successive intervals
$J_{s},J_{s+1},\ldots,J_{t}$ for some $s$ and $t$.  The following
lemma recasts~\cite[Theorem 3.3]{lpm1}.

\begin{lemma}\label{lem:order}
  For $M\in \mathcal{L}$ and a given path order of $M$, let
  $\mathcal{A}$ be the corresponding interval presentation of $M$ as
  above.  If the elements of a basis of $M$ are $x_1<x_2<\cdots<x_r$,
  then $x_i\in J_i$ for $1\leq i\leq r$.
\end{lemma}

\begin{prop}\label{prop:intervalcontract}
  Fix a path order of $M\in \mathcal{L}$; let the corresponding
  interval presentation be $\mathcal{A}$.  Assume $y\in E(M)$ is not a
  loop and let $J_{s},J_{s+1},\ldots,J_{t}$ be the sets of
  $\mathcal{A}$ that contain $y$.  The interval presentation of $M/y$
  for the induced path order is
  \begin{eqnarray*}
    \mathcal{A}' =
    \left\{
      \begin{array}{ll}
        (J_1,J_2,\ldots,J_{s-1},J_{s+1},\ldots,J_r), &\mbox{if $s=t$,}\\
        (J_1,J_2,\ldots,J_{s-1}, J'_s,J'_{s+1},\ldots,
        J'_{t-1},J_{t+1},\ldots J_r), &\mbox{if $s<t$,}
      \end{array}
    \right.  \label{fXp}
  \end{eqnarray*} 
  where $J'_i = (J_i\cup J_{i+1})-y$.  Also, if $x\in E(M) - (J_s\cap
  J_{s+1}\cap\cdots\cap J_t)$, then $x$ is in the same number of sets
  in $\mathcal{A}'$ as in $\mathcal{A}$.
\end{prop}

\begin{proof}
  The bases of $M/y$ are the sets $B\subseteq E(M)-y$ for which $B\cup
  y$ is a basis of $M$, so we are claiming that $B\cup y$ is a
  transversal of $\mathcal{A}$ if and only if $B$ is a transversal of
  $\mathcal{A}'$.  The case $s=t$ is immediate, so assume $s<t$.  Let
  $B=\{x_1,x_2,\ldots,x_{r-1}\}$ with
  $x_1<x_2<\cdots<x_k<y<x_{k+1}<\cdots<x_{r-1}$.  By
  Lemma~\ref{lem:order}, if $B\cup y$ is a transversal of
  $\mathcal{A}$, then (a) $x_i\in J_i$ for $1\leq i\leq k$, (b) $y\in
  J_{k+1}$, and (c) $x_i\in J_{i+1}$ for $k+1\leq i\leq r-1$.  Thus,
  $x_i\in J'_i$ for $s\leq i<t$, so $B$ is a transversal of
  $\mathcal{A}'$.  The converse follows with a similar argument upon
  noting that if $x_i<y$ and $x_i\in J'_h$, then $x_i\in J_h$ (note
  that $[a_{h+1},y]\subsetneq [a_h,y]$); likewise, if $y<x_i$ and
  $x_i\in J'_h$, then $x_i\in J_{h+1}$; thus, $y$ can represent
  $J_{k+1}$.  The last assertion is immediate.
\end{proof}

Given a presentation $\mathcal{A}$ of a transversal matroid $M$, we
get a presentation of $M\del e$ from $\mathcal{A}$ by removing $e$
from all sets.  For $M\in\mathcal{L}$, some adjustment may be needed
so that the presentation of $M\del e$ is an antichain; for our work,
it suffices to treat this when $e$ is not a loop and $e$ is either the
least element, $e_1$, or the greatest element, $e_n$.  If
$J_1=\{e_1\}$, then $(J_2,J_3,\ldots,J_r)$ is the interval
presentation of $M\del e_1$ for the induced path order.  If
$\{e_1\}\subsetneq J_1$, then $(J_1-e_1,J_2-e_2,\ldots, J_r-e_r)$ is
the interval presentation of $M\del e_1$ since, by
Lemma~\ref{lem:order}, $e_i$, the $(i-1)$-st element of $E(M\del
e_1)$, is not needed in the $i$-th set; note that, in this case, if
$e_i\in J_i$ with $1<i\leq r$, then $e_i$ is the lower endpoint of
$J_{i-1}-e_{i-1}$.  The interval presentation of $M\del e_n$ is
obtained similarly.

\begin{lemma}\label{lem:samenum}
  For $M\in \mathcal{L}$, let $x$ be in an interval $I$ in a given
  path order of $M$.  Let $\mathcal{A}$ be the corresponding
  interval presentation of $M$ and let $\mathcal{A}'$ be the induced
  interval presentation of the restriction $M|I$.  Either $x$ is an
  upper or lower endpoint of some interval in $\mathcal{A}'$ or $x$ is
  in the same number of intervals in $\mathcal{A}'$ as in
  $\mathcal{A}$.
\end{lemma}

The next result recasts~\cite[Theorem~5.3]{lpm2}.  For a path order
$e_1<e_2<\cdots<e_n$ of $M$, the predecessor function
$p:E(M)-e_1\rightarrow E(M)-e_n$ and successor function
$s:E(M)-e_n\rightarrow E(M)-e_1$ are defined by $p(e_i)=e_{i-1}$ and
$s(e_i)=e_{i+1}$.

\begin{prop}\label{prop:fundfl}
  Let $M\in \mathcal{L}$ be connected; let a path order of $M$ and the
  corresponding interval presentation $\mathcal{A}$ be as above.  The
  fundamental flats of $M$ are the intervals
  \begin{itemize}
  \item[(i)] $[e_1,p(a_{j+1})]$ where $p(a_{j+1})>a_j$ (which has rank
    $j$) and
  \item[(ii)] $[s(b_k),e_n]$ where $s(b_k)<b_{k+1}$ (which has rank
    $r-k$).
  \end{itemize}
\end{prop}

\begin{cor}\label{cor:keepch}
  Assume $M\in \mathcal{L}$ is connected.  Let $e$ be a terminal
  element of $M$.  Let the fundamental flats of $M$ that contain $e$
  be $F_1\subsetneq F_2\subsetneq\cdots\subsetneq F_h$.  If
  $r(F_1)>1$, then $M/e$ is connected and $F_1-e, F_2-e,\ldots, F_h-e$
  are fundamental flats of $M/e$.
\end{cor}

The following result is~\cite[Corollary~5.8]{lpm2}.

\begin{prop}\label{prop:lpmauto}
  The automorphisms of a connected matroid in $\mathcal{L}$ are the
  permutations of the ground set that are rank-preserving bijections
  of the collection of fundamental flats.
\end{prop}

The following result \cite[Theorem~5.10]{lpm2} characterizing
connected lattice path matroids plays a key role in our work.  We use
$\eta$ for the nullity function: $\eta(X)=|X|-r(X)$.

\begin{prop}\label{prop:lpmchar}
  A connected matroid $M$ is in $\mathcal{L}$ if and only if the
  properties below hold.
  \begin{itemize}
  \item[(i)] The fundamental flats of $M$ form at most two disjoint
    chains under inclusion, say $F_1\subsetneq
    F_2\subsetneq\cdots\subsetneq F_h$ and $G_1\subsetneq
    G_2\subsetneq\cdots\subsetneq G_k$.
  \item[(ii)] If $F_i\cap G_j\ne \emptyset$, then $F_i\cup G_j =
    E(M)$.
  \item[(iii)] The pnc-flats of $M$ other than $F_1, F_2,\ldots, F_h,
    G_1, G_2,\ldots, G_k$ are the intersections $F_i\cap G_j$ where
    $\eta(M)<\eta(F_i)+\eta(G_j)$.
  \item[(iv)] If $F_i\cap G_j$ is a pnc-flat, then $r(F_i\cap
    G_j)=r(F_i)+r(G_j)-r(M)$.
  \end{itemize}
\end{prop}

Note that property (ii) precludes any inclusion among any fundamental
flats $F_i$ and $G_j$.

\begin{cor}\label{cor:lpmff=irr}
  The fundamental flats of a connected matroid in $\mathcal{L}$ are
  its irreducible pnc-flats.
\end{cor}

\begin{cor}\label{cor:notlpm}
  Let $F$ and $G$ be pnc-flats of $M\in\mathcal{L}$ that are not
  disjoint.  If $F\cup G$ spans $M$, then $F\cup G = E(M)$.
\end{cor}

The following result is a mild but useful extension
of~\cite[Theorem~3.3]{lpm2}.

\begin{prop}\label{prop:spanning}
  Let $M\in\mathcal{L}$ be connected and nontrivial.  Fix an interval
  presentation $\mathcal{A}$ of $M$ as above.  If $x$ is in at least
  two sets in $\mathcal{A}$ or $x\in\{a_1,b_r\}$, then $x$ is in a
  spanning circuit of $M$.
\end{prop}

\begin{proof}
  Since $M$ is connected, $a_h\in J_{h-1}\cap J_h$ if $h>1$; also,
  $b_k\in J_k\cap J_{k+1}$ if $k<r$.  Thus, if $x\in J_i\cap J_{i+1}$,
  then each $r$-subset of
  $C=\{a_1,a_2,\ldots,a_i,x,b_{i+1},\ldots,b_r\}$ is a transversal of
  $\mathcal{A}$ and hence a basis of $M$, so $C$ is a spanning
  circuit.
\end{proof}

Note that $M\del x$ is connected if some spanning circuit of $M$ does
not contain $x$.  This applies if $x\in F_1-\{a_1,a_2,\ldots,a_r\}$
where $F_1$ is the smallest fundamental flat that contains the least
element $e_1$.  If all pairs of incomparable fundamental flats of $M$
are disjoint, then, by Proposition~\ref{prop:lpmauto}, the
automorphism group of $M$ is transitive on $F_1$.  These observations
give the following result.

\begin{cor}\label{cor:safedel}
  Assume $M\in \mathcal{L}$ is connected and has at least one
  fundamental flat.  If all pairs of incomparable fundamental flats of
  $M$ are disjoint, then for any element $x$ in a smallest fundamental
  flat of $M$, the deletion $M\del x$ is connected.
\end{cor}

We will often use the following observations along with
Proposition~\ref{prop:spanning}.  If $M$ is connected and $y$ is in
the spanning circuit $C$ of $M$, then $C-y$ is a spanning circuit of
$M/y$.  Thus, if, in addition, $\cl(y)=\{y\}$, then $M/y$ is
connected.

Lemma~\ref{lem:samenum} and Proposition~\ref{prop:spanning} have the
following corollary.

\begin{cor}\label{cor:spandel}
  For $M\in \mathcal{L}$, let $x$ be in an interval $I$ in a given
  path order of $M$ and let $\mathcal{A}$ be the corresponding
  interval presentation of $M$.  If $M|I$ is connected and $x$ is in
  at least two intervals in $\mathcal{A}$, then $x$ is in a spanning
  circuit of $M|I$.
\end{cor}

This corollary applies, for instance, if $I$ is a pnc-flat since, by
Propositions~\ref{prop:fundfl} and~\ref{prop:lpmchar}, such flats are
intervals in any path order.

We will also use the following result~\cite[Corollary 5.5]{lpm2}.
 
\begin{prop}\label{prop:dualfund}
  If $M\in\mathcal{L}$, then $M^*\in\mathcal{L}$.  If $M$ is also
  connected, then the fundamental flats of $M^*$ are the set
  complements of the fundamental flats of $M$.
\end{prop}

Matroids having interval presentations $\mathcal{A}$ as above where
either $\{a_1,a_2,\ldots,a_r\}$ or $\{b_1,b_2,\ldots,b_r\}$ is an
interval in the path order are \emph{nested matroids} (called
generalized Catalan matroids in~\cite{lpm2}).  Let $\mathcal{C}$ be
the class of these matroids.  A connected matroid in $\mathcal{L}$ is
nested if and only if its fundamental flats form a chain.  The
following related result is essentially Lemma 2 of~\cite{opr}.

\begin{prop}\label{prop:chainflats}
  A loopless matroid is in $\mathcal{C}$ if and only if its pnc-flats
  form a chain.
\end{prop}

Let $P_n$ be $T_n(U_{n-1,n}\oplus U_{n-1,n})$, the truncation to rank
$n$ of the direct sum of two $n$-circuits. Thus, $P_n$ is the rank-$n$
paving matroid whose only pnc-flats are two disjoint
circuit-hyperplanes whose union is the ground set.  The following
result is from~\cite{opr}.

\begin{prop}\label{prop:oxprow}
  A matroid is in $\mathcal{C}$ if and only if it has no $P_n$-minor
  for any $n\geq 2$.
\end{prop}

\subsection{Parallel connections}
For our purposes, the next result~\cite[Proposition 7.1.13]{ox} can be
taken as the definition of the \emph{parallel connection}
$P_x(M_1,M_2)$ of matroids $M_1$ and $M_2$ using basepoint $x$.  The
special case $P_x(M,U_{1,2})$ is the \emph{parallel extension} of $M$
at $x$.

\begin{prop}\label{prop:pcbases}
  Assume that $M_1$ and $M_2$ are matroids with $E(M_1)\cap E(M_2)=
  \{x\}$ and $r_{M_1}(x)+r_{M_2}(x)>0$.
  \begin{enumerate}
  \item A set $B\subseteq E(M_1)\cup E(M_2)$ with $x\in B$ is a basis
    of $P_x(M_1,M_2)$ if and only if $B\cap E(M_1)$ is a basis of
    $M_1$ and $B\cap E(M_2)$ is a basis of $M_2$.
  \item A set $B\subseteq E(M_1)\cup E(M_2)$ with $x\not\in B$ is a
    basis of $P_x(M_1,M_2)$ if and only if, for either $(i,j)=(1,2)$
    or $(i,j)=(2,1)$, the set $B\cap E(M_i)$ is a basis of $M_i$ and
    $\bigl(B\cap E(M_j)\bigr)\cup x$ is a basis of $M_j$.
  \end{enumerate}
\end{prop}

We will use the result below~\cite[Proposition 7.1.15]{ox} on minors
of parallel connections.

\begin{prop}\label{prop:pcminors}
  For $y\in E(M_1)- x$, we have $P_x(M_1,M_2)\del y = P_x(M_1\del
  y,M_2)$ and $P_x(M_1,M_2)/y = P_x(M_1/y,M_2)$.  Also,
  $P_x(M_1,M_2)/x = (M_1/x)\oplus (M_2/x)$.
\end{prop}

Clearly $P_x(M_1,M_2) =P_x(M_2,M_1)$, so the analogous results hold
for $y\in E(M_2)-x$.

The next result~\cite[Theorem 7.1.16]{ox} gives an important link
between connectivity and parallel connection.

\begin{prop}\label{prop:discon} 
  Let $M$ be a connected matroid with $x\in E(M)$.  If $M/x=M_1\oplus
  M_2$, then $M = P_x\bigl(M\del E(M_2),M\del E(M_1)\bigr)$;
  furthermore, both $M\del E(M_2)$ and $M\del E(M_1)$ are connected.
\end{prop}

\section{The Excluded Minors of Lattice Path
  Matroids}\label{sec:exlpm}

The excluded minors of $\mathcal{L}$ are given in
Theorem~\ref{thm:list} and illustrated in Figure~\ref{exminfig}.  (All
appeared in~\cite{lpm2}.)  That these matroids are not in
$\mathcal{L}$ follows readily from Proposition~\ref{prop:lpmchar};
checking that their proper minors are in $\mathcal{L}$ is not
difficult.  Thus, we focus on proving that these are the only excluded
minors.

We start with several points of notation.  The free extension and
coextension of $M$ by $e$ are denoted $M+e$ and $M\times e$,
respectively.  Besides the matroids $P_n$ in
Proposition~\ref{prop:oxprow}, a family of matroids that plays an
important role in this work is $P'_n= P_{n-1}\times e$ for $n\geq 3$.
Equivalently, $P'_n= T_n\bigl(P_e(U_{n-1,n},U_{n-1,n})\bigr)$, the
truncation to rank $n$ of the parallel connection of two $n$-circuits.

\begin{thm}\label{thm:list}
  A matroid is a lattice path matroid if and only if it has none of
  the following matroids as minors:
  \begin{enumerate}
  \item $A_n = P'_n+x$, for $n\geq 3$,
  \item $B_{n,k}=T_n(U_{n-1,n}\oplus U_{n-1,n}\oplus U_{k-1,k})$ and
    its dual $C_{n+k,k}$, for $n\geq k\geq 2$,
  \item $D_n = (P_{n-1}\oplus U_{1,1})+x$ and its dual $E_n$, for
    $n\geq 4$,
  \item the rank-$3$ wheel, $\mathcal{W}_3$, the rank-$3$ whirl,
    $\mathcal{W}^3$, and
  \item the matroid $R_3$ and its dual $R_4$ (see
    Figure~\ref{exminfig}).
  \end{enumerate}
\end{thm}

\begin{figure}
\begin{center}
\includegraphics[width = 4.75truein]{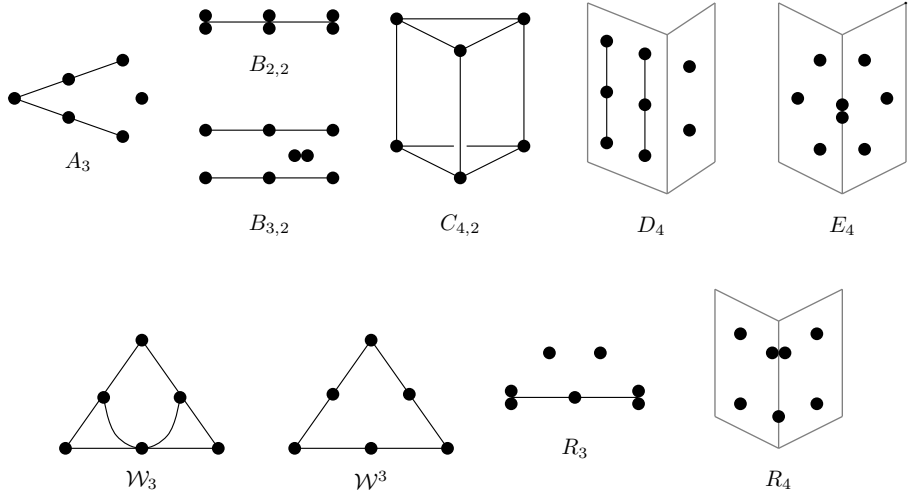}
\end{center}
\caption{The excluded minors of lattice path matroids.  Those in the
  top row are in infinite families of excluded
  minors.}\label{exminfig}
\end{figure}

Note that $A_n$ is self-dual.  The matroid $C_{n+k,k}$ is a paving
matroid of rank $n+k$; its ground set can be partitioned into sets
$X,Y,Z$ with $|X| = |Y| = n$ and $|Z| = k$ so that the only nontrivial
hyperplanes are $X\cup Y$, $X\cup Z$, and $Y\cup Z$, two (or all, if
$n=k$) of which are circuits.  In $E_n$, the element $x$ is in a
$2$-circuit and $E_n\del x=P'_n$

Let $\mathcal{E_L}$ be the set of excluded minors of $\mathcal{L}$ and
let $\mathcal{E}$ be the set of those in items (1)--(5) of
Theorem~\ref{thm:list}.  Since $\mathcal{L}$ is closed under direct
sums, matroids in $\mathcal{E_L}$ are connected.  Since $\mathcal{L}$
is closed under duality, $M\in \mathcal{E_L}$ if and only if $M^*\in
\mathcal{E_L}$.  We prove $\mathcal{E_L}-\mathcal{E} = \emptyset$ in
the subsections below.  We first describe these subsections, thereby
outlining the proof.
\begin{itemize}
\item[(3.1)] For $M_1,M_2\in\mathcal{L}$ with $E(M_1)\cap
  E(M_2)=\{x\}$, we determine whether $P_x(M_1,M_2)$ is in
  $\mathcal{L}$; we show that unless some simple sufficient conditions
  are met, $P_x(M_1,M_2)$ has one of $B_{n,2}$, $C_{4,2}$, $E_n$,
  $R_3$, or $R_4$ as a minor.
\item[(3.2)] Using these results and duality, we show that for any
  $M\in\mathcal{E_L}-\mathcal{E}$, the following minors are connected:
  $M\del x$, $M/x$, and $M\del x/y$ for all $x,y\in E(M)$.  Through
  duality, $C_{n+2,2}$ and $D_n$ enter; the matroids $A_3$,
  $\mathcal{W}_3$, and $\mathcal{W}^3$ also arise.
\item[(3.3)] We prove counterparts of Corollary~\ref{cor:lpmff=irr}
  and the second part of Proposition~\ref{prop:dualfund} for any
  $M\in\mathcal{E_L}-\mathcal{E}$.  We show how the fundamental flats
  of $M$ correspond to those of its single-element deletions and
  contractions.  We also show that $\mathcal{W}_3$, $\mathcal{W}^3$,
  $B_{n,k}$, and $C_{n+k,k}$ are the only excluded minors having three
  or more mutually incomparable fundamental flats.
\item[(3.4)] Using the results proven in the first three subsections,
  we show that the properties in Proposition~\ref{prop:lpmchar} hold
  for any $M\in\mathcal{E_L}-\mathcal{E}$, which gives the
  contradiction $M\in \mathcal{L}$, so $\mathcal{E_L}-\mathcal{E} =
  \emptyset$.  (The matroid $A_n$ appears at this stage.)
\end{itemize}

\subsection{Parallel connections}

\begin{lemma}\label{lem:lpmpc1}
  Let $M_1,M_2\in \mathcal{L}$ be nontrivial and connected matroids of
  positive rank with $E(M_1)\cap E(M_2)= \{x\}$.
\begin{itemize}
\item[(i)] If $x$ is a terminal element of both $M_1$ and $M_2$, then
  $P_x(M_1,M_2) \in \mathcal{L}$.
\item[(ii)] If $M_1$ is a parallel connection using basepoint $x$,
  then $P_x(M_1,U_{1,2})\in \mathcal{L}$.
\end{itemize}
\end{lemma}

\begin{proof}
  For assertion (i), fix path orders $e_1<\cdots<e_m<x$ and
  $x<f_1<\cdots<f_n$ of $M_1$ and $M_2$, respectively, with the
  corresponding interval presentations $\mathcal{A}_1 =
  (J_1,\ldots,J_{r_1})$ and $\mathcal{A}_2 = (J'_1,\ldots,J'_{r_2})$.
  Thus, $x$ is the upper endpoint of $J_{r_1}$ and the lower endpoint
  of $J'_1$.  It follows from Proposition~\ref{prop:pcbases} that
  $(J_1,\ldots,J_{r_1-1},J_{r_1}\cup J'_1,J'_2,\ldots,J'_{r_2})$ is
  the interval presentation of $P_x(M_1,M_2)$ for the path order
  $e_1<\cdots<e_m<x<f_1<\cdots<f_n$.

  To prove assertion (ii), we first claim that $x$ is in only one set
  in any interval presentation $\mathcal{A}$ of $M_1$.  If $x$ is in a
  $2$-circuit, then the claim follows from general considerations
  about transversal matroids; otherwise, all components of the
  disconnected matroid $M_1/x$ have positive rank, so the claim
  follows from Proposition~\ref{prop:spanning}.  To get an interval
  presentation of the parallel extension of $M_1$ by $y$, insert $y$
  immediately after $x$ in the path order of $M_1$ and adjoin $y$ to
  the only interval in $\mathcal{A}$ that contains $x$.
\end{proof}

The following corollary of the proof above can also be shown using
Proposition~\ref{prop:intervalcontract}.

\begin{cor}\label{cor:charpc}
  Assume $M\in \mathcal{L}$ is connected and, in a given path order,
  $x$ is neither the first nor the last element.  There are
  $M_1,M_2\in\mathcal{L}$ with $M = P_x(M_1,M_2)$ if and only if $x$
  is in just one set of the interval presentation.
\end{cor}

\begin{lemma}\label{lem:ter}
  Let $M \in \mathcal{L}$ be connected.  Let $S_1$ and $S_2$ be proper
  subsets of $E(M)$ with $S_1\cap S_2 = \{x\}$ and
  $M=P_x(M|S_1,M|S_2)$.  If $x$ is in no $2$-circuit of $M$, then $x$
  is a terminal element of both $M|S_1$ and $M|S_2$.
\end{lemma}

\begin{proof}
  Fix a path order $e_1<e_2<\cdots<e_n$ of $M$.  Since $\cl(x) =
  \{x\}$, both $S_1$ and $S_2$ are pnc-flats of $M$.  Since $S_1\cup
  S_2 = E(M)$ and $S_1\cap S_2 = \{x\}$, the description of pnc-flats
  given in Propositions~\ref{prop:fundfl} and~\ref{prop:lpmchar}
  implies that $S_1$ and $S_2$ are, in some order, $[e_1,x]$ and
  $[x,e_n]$, so $x$ is terminal in $M|S_1$ and $M|S_2$.
\end{proof}

\begin{lemma}\label{lem:lpmpc2}
  Assume $E(M_1)\cap E(M_2)= \{x\}$ for nontrivial connected matroids
  $M_1,M_2$ in $\mathcal{L}$ of positive rank.  Assume $x$ is
  nonterminal in $M_1$; if $M_1/x$ is disconnected, then also assume
  $r(M_2)>1$.  At least one of $B_{n,2}$, $C_{4,2}$, $E_n$, $R_3$,
  $R_4$ is a minor of $P_x(M_1,M_2)$.
\end{lemma}

\begin{proof}
  If $\{x,y\}$ is a circuit of $M_1$, then since $x$ is nonterminal in
  $M_1$, it is nonterminal in $M_1\del y$.  Thus, it suffices to prove
  the result when $\cl_{M_1}(x)=\{x\}$.
   
  Assume $M_1/x$ is disconnected.  Thus, $M_1$ is the parallel
  connection, at $x$, of two connected matroids, each of rank at least
  two since $\cl_{M_1}(x)=\{x\}$, so, by
  Proposition~\ref{prop:pcminors}, $M_1$ has a $P'_3$-minor with $x$
  in both $3$-circuits.  Now $r(M_2)>1$, so $P_x(M_1,M_2)$ has, as a
  minor, the parallel connection of three $3$-circuits with the
  basepoint $x$; deleting $x$ from this minor yields $C_{4,2}$.

  Now assume $M_1/x$ is connected.  Fix a path order of $M_1$.  If
  $M_1\in \mathcal{C}$ and $F_1\subsetneq \cdots\subsetneq F_h$ are
  its fundamental flats, then $F_1\cup \bigl(E(M_1)-F_h\bigr)$ is its
  set of terminal elements.  If $M_1\not\in\mathcal{C}$ and
  $F_1\subsetneq \cdots\subsetneq F_h$ and $G_1\subsetneq
  \cdots\subsetneq G_k$ are its fundamental flats, then, by
  Proposition~\ref{prop:lpmauto}, its set of terminal elements is
  $(F_1-G_k)\cup (G_1-F_h)$.  Thus, by symmetry, we may assume one of
  the following options holds: (a) $x\in F_i-F_{i-1}$ for some $i$
  with $1<i\leq h$, (b) $x\in F_1\cap G_1$, or (c) $x\not\in F_h\cup
  G_k$.

  \begin{figure}
    \begin{center}
      \includegraphics[width =2.0truein]{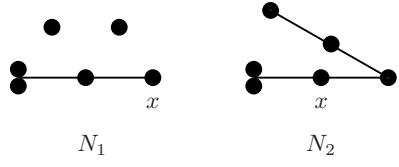}
    \end{center}
    \caption{Two minors that arise in the proof of
      Lemma~\ref{lem:lpmpc2}.}\label{r3etc2}
  \end{figure}

  Assume $x\in F_i-F_{i-1}$.  Among all minors of $M_1$ that meet the
  following conditions, let $N$ be one for which $|E(N)|$ is minimal:
  (a) $x\in E(N)$, (b) $N$ and $N/x$ are connected, and (c) for at
  least one of the chains $F'_1\subsetneq \cdots\subsetneq F'_t$ of
  fundamental flats of $N$, we have $x \in F'_s-F'_{s-1}$ for some $s$
  with $1<s\leq t$.  We claim that $N$ is one of the matroids in
  Figure~\ref{r3etc2}.  To see this, first note that, by
  Corollary~\ref{cor:charpc}, since $N/x$ is connected, $x$ is in at
  least two intervals in the induced interval presentation
  $\mathcal{A}'$ of $N$.  We may assume $F'_1$ contains the least
  element, $a$, of $E(N)$.  If $r(F'_1)>1$, then, by
  Corollary~\ref{cor:keepch}, $N/a$ would contradict the minimality of
  $|E(N)|$ (note that $x$ would be in at least two intervals in the
  interval presentation of $N/a$, so $N/a,x$ would be connected);
  thus, $r(F'_1)=1$.  If either $|F'_1|>2$ or $s>2$, then $N\del a$
  would contradict the minimality of $|E(N)|$, so $F'_1$ is a
  $2$-circuit and $s=2$.  We claim $r(F'_2) = 2$.  Assume $r(F'_2) >
  2$ and consider the intervals $J_2=[a_2,b_2]$ and $J_3=[a_3,b_3]$ of
  $\mathcal{A}'$.  If $a_2\ne x$, then, since $N$ is connected, either
  $J_1\cap\{a_2,x\}=\{a_2\}$ or $x$ is in at least three intervals;
  thus, by Proposition~\ref{prop:intervalcontract}, $x$ is in at least
  two intervals in the presentation of $N/a_2$; spanning circuits show
  that $N/a_2$ is connected; these conclusions contradict the
  minimality of $|E(N)|$, so $a_2=x$.  Thus $x\not\in J_3$, so $N/a_3$
  is connected and has $x$ in at least two presentation intervals,
  which contradicts the minimality of $|E(N)|$.  Thus, $F'_2$ is a
  line.  The minimality of $|E(N)|$ also gives $|F'_2|=4$. If
  $r(N)>3$, and $b$ is the greatest element of $E(N)$, then $N\del
  (\cl_N(b)-b)/b$ would contradict the minimality of $|E(N)|$, so
  $r(N) = 3$.  Similar arguments show that $E(N)-F'_2$ is an
  independent set of size two.  Since $N/x$ is connected,
  $(E(N)-F'_2)\cup x$ is not a line.  Thus, $N$ is either $N_1$ or
  $N_2$ of Figure~\ref{r3etc2}.  If $N=N_1$, then, by
  Lemma~\ref{prop:pcminors}, any parallel connection with $M_1$ at $x$
  has an $R_3$-minor; if $N=N_2$, then any such parallel connection
  has a $B_{2,2}$-minor.

  Now assume $x\in F_1\cap G_1$.  Among all minors of $M_1$ that meet
  the following conditions, let $N$ be one for which $|E(N)|$ is
  minimal: (a) $x\in E(N)$, (b) $N$ and $N/x$ are connected, (c) not
  all fundamental flats of $N$ are comparable, and (d) $x$ is in all
  fundamental flats of $N$.  We claim that $N$ is either $P'_n$, for
  some $n\geq 4$, or the simplification of $R_4$.  Let $F$ (resp.,
  $G$) be the smallest fundamental flat that contains the least
  (resp., greatest) element of $E(N)$. Property (d) implies that $x$
  is in all pnc-flats, so, by property (b), $N$ has no $2$-circuits.
  If $r(F)\leq r(N)-2$, then $N/b$, where $b$ is the greatest element
  of $E(N)$, would contradict the minimality of $|E(N)|$, so $F$ (and
  likewise $G$) is a hyperplane of $N$.  If $r(F)=r(G)=2$, then $N/x$
  would be disconnected (note that $F\cup G = E(N)$), so $r(N)\geq 4$.
  Since $F$ and $G$ are the only fundamental flats of $N$, by
  Proposition~\ref{prop:lpmchar}, the only possible pnc-flat of $N$
  besides $F$ and $G$ is $F\cap G$.  By Corollary~\ref{cor:charpc} and
  Proposition~\ref{prop:spanning}, there is a spanning circuit $C$ of
  $N$ with $x\in C$.  If $F\cap G$ were a pnc-flat, then $F\cap
  G\not\subseteq C$ and both $N|F$ and $N|G$ would be non-uniform
  nested matroids; it follows that $N\del y$, for any $y\in (F\cap
  G)-C$, would contradict the minimality of $|E(N)|$.  Thus, $F$ and
  $G$ are the only pnc-flats of $N$.  Since $N|F$ and $N|G$ are
  uniform, by the minimality of $|E(N)|$, both $F$ and $G$ are
  circuits.  Assume first $|F\cap G|=r(N)-2$, so $|F-G| = 2 = |G-F|$.
  If $r(N)>4$, then $N/y$, for any $y\in (F\cap G)-x$, would
  contradict the minimality of $|E(N)|$.  Thus, $r(N)=4$ and $N$ is
  the simplification of $R_4$, with $x$ in both $4$-circuits;
  therefore any parallel connection using $M_1$ with $x$ as the
  basepoint has an $R_4$-minor.  (To prepare for the next paragraph,
  note that the dual of this minor $N$ is a line with four points, two
  of which are $2$-circuits, and $x$ is not in a $2$-circuit.)  Now
  assume $|F\cap G|<r(N)-2$, so $|F-G|\geq 3$ and $|G-F|\geq 3$.  The
  minimality of $|E(N)|$ forces $F\cap G = \{x\}$, so $N = P'_n$ for
  some $n\geq 4$, with $x$ being common to the two nonspanning
  circuits.  In this case, any parallel connection using $M_1$ with
  $x$ as the basepoint has an $E_n$-minor for some $n\geq 4$.

  Finally, assume $x\not\in F_h\cup G_k$.  Using
  Proposition~\ref{prop:dualfund}, it follows that $x$ is in all
  fundamental flats of $M_1^*$.  Therefore, by the results in the last
  paragraph, $M_1$ has, as a minor, either (a) a $4$-point line with
  two $2$-circuits, neither of which contains $x$ or (b) the dual of
  $P'_n$ for some $n\geq 4$, with $x$ in neither circuit-hyperplane.
  It follows that any parallel connection using $M_1$ with $x$ as the
  basepoint has, in the first case, a $B_{2,2}$-minor and, in the
  second case, a $B_{n,2}$-minor with $n\geq 3$.
\end{proof}

\subsection{Connectivity}

Recall that $M\in \mathcal{E_L}-\mathcal{E}$ if and only if $M$ is an
excluded minor of $\mathcal{L}$ that is not in items (1)--(5) of
Theorem~\ref{thm:list}.

\begin{lemma}\label{lem:minorcon}
  For $M\in\mathcal{E_L}-\mathcal{E}$, both $M\del x$ and $M/x$ are
  connected for all $x\in E(M)$.
\end{lemma}

\begin{proof}
  We argue by contradiction.  Assume $M/x$ is disconnected.  By
  Proposition~\ref{prop:discon}, there are subsets $S_1, S_2\subsetneq
  E(M)$ with $M=P_x(M|S_1,M|S_2)$ where $M|S_1$ and $M|S_2$ are
  connected.  By Lemma~\ref{lem:lpmpc1}, since
  $M|S_1,M|S_2\in\mathcal{L}$ yet $M\not\in\mathcal{L}$, we may assume
  $x$ is not terminal in $M|S_1$; also, if $M|S_1/x$ is disconnected,
  then we may assume $r(S_2)>1$.  From Lemma~\ref{lem:lpmpc2}, some
  minor of $M$ is in $\mathcal{E}$, contrary to $M\in
  \mathcal{E_L}-\mathcal{E}$.  Thus, $M/x$ is connected.  That $M\del
  x$ is connected follows since $M\del x=(M^*/x)^*$ and
  $M^*\in\mathcal{E_L}-\mathcal{E}$.
\end{proof}

\begin{cor}\label{cor:no2}
  Matroids in $\mathcal{E_L}-\mathcal{E}$ have no $2$-circuits and no
  $2$-cocircuits.
\end{cor}

\begin{lemma}\label{lem:minorcon2}
  If $M\in\mathcal{E_L}-\mathcal{E}$, then $M\del x/y$ is connected
  for all $x,y\in E(M)$.
\end{lemma}

\begin{proof}
  Assume $M\del x/y$ is disconnected; we will get the contradiction
  that $M$ is $A_3$, $\mathcal{W}_3$, or $\mathcal{W}^3$.  Since
  $M\del x$ is connected, $M\del x=P_y(M|S_1,M|S_2)$ for some proper
  subsets $S_1$, $S_2$ of $E(M)-x$ where $M|S_1$ and $M|S_2$ are
  connected.  Lemma~\ref{lem:ter} and Corollary~\ref{cor:no2} imply
  that $y$ is terminal in $M|S_1$ and $M|S_2$.  Now $M/y\del x =
  (M|S_1/y)\oplus (M|S_2/y)$ yet $M/y$ is connected, so $x\not\in
  \cl_M(S_1)\cup \cl_M(S_2)$.  Since $M|S_1$ and $M|S_2$ are connected
  and in $\mathcal{L}$, and since $y$ is terminal in both, some
  spanning circuits $C_1$ of $M|S_1$ and $C_2$ of $M|S_2$ contain $y$.
  Now $|C_1|\geq 3$ and $|C_2|\geq 3$ by Corollary~\ref{cor:no2}.
  Since $C_1\cup C_2$ spans $M$ but $x\not\in \cl_M(C_1)\cup
  \cl_M(C_2)$, Corollary~\ref{cor:notlpm} gives $E(M)=C_1\cup C_2\cup
  x$, so $M$ is a single-element extension of $P_y(M|C_1,M|C_2)$.
  Semimodularity applied to $C_1$ and $\cl_M(C_2\cup x)$ gives
  $r\bigl(C_1\cap \cl_M(C_2\cup x)\bigr)\leq 2$.  If $|C_1|>3$, then
  $M/z$, for $z \in (C_1-y)-\cl_M(C_2\cup x)$, would be a
  single-element extension, by $x$, of $P_y(M|C_1/z,M|C_2)$; both
  $C_1-z$ and $C_2$ would be pnc-flats of $M/z$ yet $x\not\in
  (C_1-z)\cup C_2$, contrary to Corollary~\ref{cor:notlpm}.  Thus,
  $M\del x=P'_3$, so, as claimed, $M$ is $A_3$, $\mathcal{W}_3$, or
  $\mathcal{W}^3$.
\end{proof}

\subsection{Fundamental flats}
The following four lemmas enter into the proof of
Lemma~\ref{lem:f=ir}, which is a counterpart of
Corollary~\ref{cor:lpmff=irr}.

\begin{lemma}\label{lem:ff1}
  For a connected matroid $M$ and connected deletion $M\del x$, if $F$
  is a fundamental flat of $M\del x$, then $\cl_M(F)$, which is $F$ or
  $F\cup x$, is a fundamental flat of $M$.
\end{lemma}

\begin{proof}
  The spanning circuit $C$ of $M\del x$ that shows that $F$ is a
  fundamental flat of $M\del x$ also shows that $\cl_M(F)$ is a
  fundamental flat of $M$.
\end{proof}

\begin{lemma}\label{lem:prepnc}
  Assume $x$ is not a loop of $M$.  If $F$ is a pnc-flat of $M\del x$,
  then exactly one of $F$ and $F\cup x$ is a pnc-flat of $M$.  The
  same conclusion holds if $F$ is a pnc-flat of $M/x$.
\end{lemma}

\begin{proof}
  The first assertion is evident since $\cl_M(F)$ is either $F$ or
  $F\cup x$ and $x$ is not a loop.  For the second, note that $F\cup
  x$ is a flat of $M$ since $F$ is a flat of $M/x$.  If $\cl_M(F)=F$,
  then $x$ is an isthmus of $M|F\cup x$; thus, $M|F=M/x|F$, so $F$ is
  a pnc-flat of $M$.  Assume $\cl_M(F)=F\cup x$.  Now $x$ is not a
  component of $M|F\cup x$ but $M|(F\cup x)/x = M/x|F$, which we
  assumed is connected.  Thus, $M|F\cup x$ is connected, so $F\cup x$
  is a pnc-flat.
\end{proof}

\begin{lemma}\label{lem:staycon}
  For $M\in\mathcal{E_L}-\mathcal{E}$, if $F$ is a pnc-flat of $M$ and
  $y\in F$, then $F-y$ is a pnc-flat of $M/y$.  Furthermore, $F-y$ is
  reducible in $M/y$ if and only if $F$ is reducible in $M$.
\end{lemma}

\begin{proof}
  For the first part, we need to show that $M|F/y$ is connected.  Fix
  $x \in E(M)-F$.  Take a path order of $M\del x$ and the
  corresponding interval presentation $\mathcal{A}$.  By
  Lemma~\ref{lem:minorcon2}, $M\del x/y$ is connected, so $y$ is
  either a terminal element or in at least two sets in $\mathcal{A}$.
  Thus, $y$ is in a spanning circuit of $M|F$ by
  Corollary~\ref{cor:spandel}, so $M|F/y$ is connected.

  For the second assertion, first assume $F$ is reducible in $M$, so
  $F=G\cap H$ for some incomparable pnc-flats $G$ and $H$ of $M$.  As
  just shown, $G-y$ and $H-y$ are pnc-flats of $M/y$, so their
  intersection, $F-y$, is reducible in $M/y$.  Now assume $F-y$ is
  reducible in $M/y$, so $F-y=G\cap H$ for some incomparable pnc-flats
  $G$ and $H$ of $M/y$.  Since $y\in \cl_M(F-y)$, by
  Lemma~\ref{lem:prepnc} both $G\cup y$ and $H\cup y$ are pnc-flats of
  $M$, so their intersection, $F$, is reducible.
\end{proof}

The same argument proves the next lemma.

\begin{lemma}\label{lem:stayconlp}
  Fix $y\in E(M)$ where $M\in \mathcal{L}$ and both $M$ and $M/y$ are
  connected.  If $F$ is a reducible pnc-flat of $M$ with $y\in F$,
  then $F-y$ is a reducible pnc-flat of $M/y$.
\end{lemma}

\begin{lemma}\label{lem:f=ir}
  A pnc-flat $F$ of $M\in\mathcal{E_L}-\mathcal{E}$ is fundamental if
  and only if it is irreducible.
\end{lemma}

\begin{proof}
  Assume $F$ is fundamental in $M$.  Thus, $M$ has a spanning circuit
  $C$ so that $F\cap C$ is a basis of $F$.  Fix $y\in F\cap C$.  By
  Lemma~\ref{lem:staycon}, $F-y$ is a pnc-flat of $M/y$.  Now $C-y$ is
  a spanning circuit of $M/y$ and $(C-y)\cap (F-y)$ is a basis of
  $F-y$ in $M/y$, so $F-y$ is a fundamental flat of $M/y$.  Since
  $M/y\in \mathcal{L}$, it follows that $F-y$ is irreducible in $M/y$.
  Therefore, by Lemma~\ref{lem:staycon}, $F$ is irreducible in $M$.

  Now assume $F$ is irreducible in $M$.  Fix $y\in F$.  The
  irreducible pnc-flat $F-y$ of $M/y$ is fundamental by
  Corollary~\ref{cor:lpmff=irr}.  By Proposition~\ref{prop:fundfl}, we
  may assume the first element, $e_1$, in a given path order of $M/y$
  is in $F-y$. Fix $x\not\in F$.  Since the pnc-flat $F-y$ of $M/y\del
  x$ contains $e_1$, it is fundamental in $M/y\del x$ by
  Proposition~\ref{prop:lpmchar}.  Thus, $F-y$ is irreducible in
  $M\del x/y$, so by Lemma~\ref{lem:stayconlp}, the pnc-flat $F$ of
  $M\del x$ is irreducible and so fundamental.  Thus, by
  Lemma~\ref{lem:ff1}, $F$ is fundamental in $M$.
\end{proof}

\begin{cor}\label{cor:ff2}
  For $M\in\mathcal{E_L}-\mathcal{E}$, if $F$ is a fundamental flat of
  $M$, then, for all $y\in F$, the set $F-y$ is a fundamental flat of
  $M/y$.
\end{cor}

\begin{lemma}\label{lem:compfun}
  For $M\in\mathcal{E_L}-\mathcal{E}$, a proper nonempty subset $F$ of
  $E(M)$ is a fundamental flat of $M$ if and only if $E(M)-F$ is a
  fundamental flat of $M^*$.
\end{lemma}

\begin{proof}
  Let $F$ be a fundamental flat of $M$.  Thus, $F$ is a cyclic flat of
  $M$, so $E(M)-F$ is a cyclic flat of $M^*$.  Fix $y\in F$.  By
  Corollary~\ref{cor:ff2}, $F-y$ is a fundamental flat of $M/y$.
  Since $M/y\in \mathcal{L}$, using Proposition~\ref{prop:dualfund},
  $E(M)-F$ is a fundamental flat of $(M/y)^*$, that is, $M^*\del y$.
  By Lemma~\ref{lem:ff1}, $\cl_{M^*}\bigl(E(M)-F\bigr)$, which is
  $E(M)-F$, is a fundamental flat of $M^*$.  The other implication
  follows by duality.
\end{proof}

\begin{cor}\label{cor:confun1}
  For $M\in\mathcal{E_L}-\mathcal{E}$, if $F$ is a fundamental flat of
  $M/x$, then exactly one of $F$ and $F\cup x$ is a fundamental flat
  of $M$.
\end{cor}

\begin{cor}\label{cor:confun2}
  For $M\in\mathcal{E_L}-\mathcal{E}$, if $F$ is a fundamental flat of
  $M$, then, for all $z\not\in F$, the set $F$ is a fundamental flat
  of $M\del z$.
\end{cor}

These results also yield a near-counterpart of property (ii) of
Proposition~\ref{prop:lpmchar}.

\begin{lemma}\label{lem:fewpoints}
  For $M\in\mathcal{E_L}-\mathcal{E}$, if $F$ and $G$ are incomparable
  fundamental flats of $M$ with $F\cap G\ne\emptyset$, then
  $|E(M)-(F\cup G)|\leq 1$; also, if $|F\cap G|\geq 2$, then
  $E(M)=F\cup G$.
\end{lemma}

\begin{proof}
  The inequality holds since if $y,z\in E(M)-(F\cup G)$, then $M\del
  z$ and its fundamental flats $F$ and $G$ would contradict property
  (ii) of Proposition~\ref{prop:lpmchar}.  Similarly, property (ii)
  applied to $M/x$, for $x\in F\cap G$, gives the second assertion.
\end{proof}

We the next lemma follows easily from the perspective of
irreducibility.

\begin{lemma}\label{lem:tocir}
  Let $F$ be a fundamental flat of $M\in\mathcal{E_L}-\mathcal{E}$.
  If $C$ is a spanning circuit of $M|F$ and $u\in F-C$, then $F-u$ is
  a fundamental flat of $M\del u$.
\end{lemma}

\begin{lemma}\label{lemma:no3}
  No three fundamental flats of $M\in\mathcal{E_L}-\mathcal{E}$ are
  mutually incomparable.
\end{lemma}

\begin{proof}
  To the contrary, assume $F_1,F_2,F_3$ are mutually incomparable
  fundamental flats of $M$.  We will derive the contradiction $M\in
  \mathcal{E}$.

  If $F_1$, $F_2$, and $F_3$ are mutually disjoint, we could work
  instead with $M^*$, in which, by Lemma~\ref{lem:compfun}, the
  complements of these sets are (non-disjoint) fundamental flats.
  Thus, we may assume $F_1\cap F_2\ne \emptyset$.  By
  Lemma~\ref{lem:fewpoints}, $|E(M)-(F_1\cup F_2)|\leq 1$, so $F_3\cap
  F_1\ne \emptyset$ and $F_3\cap F_2\ne \emptyset$.
  Corollary~\ref{cor:confun2} gives $F_1\cup F_2\cup F_3 = E(M)$, for
  otherwise deleting an element not in $F_1\cup F_2\cup F_3$ would
  give a matroid in $\mathcal{L}$ with three incomparable fundamental
  flats, which is impossible.  Similarly, $F_1\cap F_2\cap
  F_3=\emptyset$ by Corollary~\ref{cor:ff2}.
   
  Assume $|F_1\cap F_2|=1$.  The connected flat $F_1$ is not the union
  of the flat $F_1\cap F_3$ and the singleton $F_1\cap F_2$, so $|F_1
  - (F_2\cup F_3)|=1$ by Lemma~\ref{lem:fewpoints}.  Similarly, $|F_2
  - (F_1\cup F_3)|=1$.  These conclusions and
  Lemma~\ref{lem:fewpoints} give $|F_1\cap F_3|=|F_2\cap F_3|=1$, so
  $F_1$, $F_2$, and $F_3$ are $3$-circuits.  It follows that $M$ is
  either $\mathcal{W}_3$ or $\mathcal{W}^3$, contrary to
  $M\not\in\mathcal{E}$.

  Assume $|F_i\cap F_j|\geq 2$ whenever $\{i,j,k\} = \{1,2,3\}$, so
  $E(M)=F_i\cup F_j$ and $$F_i = (F_i\cap F_j)\cup (F_i\cap F_k) =
  E(M)-(F_j\cap F_k).$$ We claim that none of $F_1$, $F_2$, $F_3$ is
  properly contained in a fundamental flat, so none of them is
  properly contained in any pnc-flat.  To see this, assume, for
  instance, $F_1\subseteq F'_1$ where $F'_1$ is a fundamental flat.
  Since $F_1\cup F_i=E(M)$ for $i\in \{2,3\}$, any inclusion between
  $F'_1$ and either $F_2$ or $F_3$ would give the contradiction that
  the larger is $E(M)$.  Thus, $F'_1, F_2, F_3$ are mutually
  incomparable, so the arguments above apply to $F'_1, F_2, F_3$;
  however, this gives $F'_1=E(M)-(F_2\cap F_3)=F_1$.

  We claim that $F_1$ is a hyperplane.  To see this, fix $x\in F_2\cap
  F_3$.  Both $F_2-x$ and $F_3-x$ are fundamental flats of $M/x$ by
  Corollary~\ref{cor:ff2}.  If $F_1$ were not a hyperplane, then
  $\cl_{M/x}(F_1)$ would be a pnc-flat of $M/x$; furthermore,
  $\cl_{M/x}(F_1)$ is not properly contained in any pnc-flat of $M/x$,
  so it would be a fundamental flat of $M/x$.  However, $M/x \in
  \mathcal{L}$ cannot have three incomparable fundamental flats, so we
  may assume $F_2-x\subseteq \cl_{M/x}(F_1)$.  Since $E(M)=F_1\cup
  F_2$, we get $\cl_M(F_1\cup x)=E(M)$, so $F_1$ actually is a
  hyperplane of $M$.  By symmetry, $F_2$ and $F_3$ are also
  hyperplanes.

  We claim that $F_1$, $F_2$, $F_3$ are the only pnc-flats of $M$.  If
  such exists, consider a fundamental flat $F\not\in\{ F_1, F_2,
  F_3\}$.  If $F$ were incomparable to two of $F_1$, $F_2$, $F_3$, say
  to $F_2$ and $F_3$, then applying the arguments above to the triple
  $F,F_2,F_3$ would give the contradiction $F=E(M)-(F_2\cap F_3)=F_1$.
  Thus, any fundamental flat (and so any pnc-flat) of $M$ other than
  $F_1,F_2,F_3$ is a subset of two of these, so assume $F\subsetneq
  F_1\cap F_2$.  If such exists, let $F'\not\in\{ F_1, F_2, F_3\}$ be
  a fundamental flat of $M$ with $F'$ incomparable to $F$.  If
  $F'\not\subseteq F_1\cap F_2$, then $F\cap F'=\emptyset$; if
  $F'\subsetneq F_1\cap F_2$, then again $F\cap F'=\emptyset$ since
  $F$ and $F'$ must be fundamental flats of $M|F_1$ yet $F\cup F\ne
  F_1$.  Therefore, by Corollary~\ref{cor:safedel}, for any $x$ in a
  smallest fundamental flat in $F_1\cap F_2$, both $M|F_1\del x$ and
  $M|F_2\del x$ are connected, so $M\del x$ would have three
  incomparable fundamental flats ($F_1-x$, $F_2-x$, and $F_3$), which
  is impossible since $M\del x \in \mathcal{L}$.  Thus, $F_1$, $F_2$,
  $F_3$ are the only fundamental flats of $M$.  To see that they are
  the only pnc-flats of $M$, note that if, say, $F_1\cap F_2$ were
  connected, then, by Corollary~\ref{cor:safedel}, for any $x \in
  F_1\cap F_2$, both $M|F_1\del x$ and $M|F_2\del x$ would be
  connected, leading to the same contradiction.

  Thus, $M|F_1$, $M|F_2$, and $M|F_3$ are uniform matroids.  Note that
  at least two of $F_1$, $F_2$, $F_3$ are circuits; indeed, if, say,
  $F_1$ and $F_2$ were not circuits, then, for any $x\in F_1\cap F_2$,
  the sets $F_1-x$, $F_2-x$, and $F_3$ would be fundamental flats of
  $M\del x$, which is impossible.  It follows that $M=C_{n,k}$ where
  $n$ and $k$ are, respectively, the largest and smallest of $|F_1\cap
  F_2|$, $|F_1\cap F_3|$, $|F_2\cap F_3|$, contrary to
  $M\not\in\mathcal{E}$.
\end{proof}

\subsection{The last step}

\begin{lemma}
  All excluded minors of $\mathcal{L}$ are in $\mathcal{E}$.
\end{lemma}

\begin{proof}
  Assume, to the contrary, $M\in\mathcal{E_L}-\mathcal{E}$.  We will
  derive the contradiction $M\in\mathcal{L}$ by showing that $M$
  satisfies properties (i)--(iv) in Proposition~\ref{prop:lpmchar}.

  Not all fundamental flats of $M$ are comparable, for otherwise $M$
  would have no other pnc-flats and Proposition~\ref{prop:chainflats}
  would give the contradiction $M \in \mathcal{C}$.  Fix a fundamental
  flat $F$ of $M$.  By Lemma~\ref{lemma:no3}, the fundamental flats of
  $M$ that are incomparable to $F$ form a chain, say $G_1\subsetneq
  G_2\subsetneq\cdots\subsetneq G_k$.  Considering $M/y$ with $y\in F$
  shows that the fundamental flats that contain $F$ form a chain;
  considering $M\del x$ with $x\not\in F$ shows that those that are
  contained in $F$ form a chain; together, these give the chain of
  fundamental flats that are comparable to $F$, say $F_1\subsetneq
  F_2\subsetneq\cdots\subsetneq F_h$.  Thus, property (i) of
  Proposition~\ref{prop:lpmchar} holds.  Note that no $F_i$ is
  comparable to any $G_j$, for otherwise the same argument starting
  with $F_i$ would have the incomparable fundamental flats $F$ and
  $G_j$ in the chain of those that are comparable to $F_i$.

  To prove property (ii), by Lemma~\ref{lem:fewpoints} it suffices to
  show that having $F_i\cap G_j = \{x\}$ and $E(M)-(F_i\cup G_j) =
  \{y\}$ yields a contradiction.  Since $F_i-x$ is connected in $M/x$,
  by Corollary~\ref{cor:charpc} and Proposition~\ref{prop:spanning}
  some spanning circuit $C$ of $F_i$ contains $x$.  If $u\in F_i-C$,
  then, using Lemma~\ref{lem:tocir}, $M\del u$ with the fundamental
  flats $F_i-u$ and $G_j$ would contradict property (ii).  It follows
  that $F_i$, and likewise $G_j$, is a circuit.  We claim that both
  are also hyperplanes.  If $F_i$ were not a hyperplane, then
  $\cl_M(F_i\cup y)\ne E(M)$, so there would be a $z\in
  G_j-\cl_M(F_i\cup y)$.  Thus, $y\not\in\cl_M(F_i\cup z)$.  Thus, in
  $M/z$, the pnc-flats $G_j-z$ and $\cl_{M/z}(F_i)$ would be
  incomparable and not disjoint, yet their union would contain all
  elements except $y$, contrary to Corollary~\ref{cor:notlpm}.  Since
  no pnc-flat is comparable to either $F_i$ or $G_j$ (they are
  circuit-hyperplanes) and since no three fundamental flats are
  incomparable, there are no other fundamental flats and so no other
  pnc-flats.  Thus, $y$ is in no pnc-flat and $M\del y = P'_n$, which
  gives the contradiction $M=A_n$, so property (ii) holds.

  To prove properties (iii) and (iv), first note that since the
  fundamental flats of $M$ form two chains, the other (i.e.,
  reducible) pnc-flats are among the nonempty sets $F_i\cap G_j$.
  First assume $F_i\cap G_j$ is a pnc-flat of $M$.  Fix $x\in F_i\cap
  G_j$.  Now $(F_i\cap G_j)-x$ is a pnc-flat of $M/x$; also, $F_i-x$
  and $G_j-x$ are fundamental flats in $M/x$.  Since $M/x\in
  \mathcal{L}$, we have
  $\eta(M/x)<\eta_{M/x}(F_i-x)+\eta_{M/x}(G_j-x)$, which gives
  $\eta(M)<\eta_M(F_i)+\eta_M(G_j)$.  Property (iv) for $F_i$ and
  $G_j$ in $M$ follows from this property for $F_i-x$ and $G_j-x$ in
  $M/x$.

  Now assume $F_i\cap G_j\ne \emptyset$ and
  $\eta(M)<\eta(F_i)+\eta(G_j)$.  Since $F_i\cup G_j = E(M)$, this
  inequality can be recast as $|F_i\cup G_j|-r(F_i\cup G_j) <
  |F_i|-r(F_i) + |G_j|-r(G_j)$.  Since $F_i\cap G_j\ne\emptyset$,
  semimodularity gives $r(F_i)+r(G_j)-r(F_i\cup G_j)\geq 1$.  The last
  two inequalities give $|F_i\cap G_j|\geq 2$.  Fix $x\in F_i\cap
  G_j$.  Now $F_i-x$ and $G_j-x$ are incomparable fundamental flats of
  $M/x$ that are not disjoint; also, the assumed inequality about
  nullity gives $\eta(M/x)<\eta_{M/x}(F_i-x)+\eta_{M/x}(G_j-x)$.
  Therefore $(F_i-x)\cap (G_j-x)$ is a pnc-flat of $M/x$, so either
  $F_i\cap G_j$ or $(F_i\cap G_j)-x$ is a pnc-flat of $M$.  If
  $F_i\cap G_j$ were not a pnc-flat of $M$, then the same argument
  using some $y\in (F_i\cap G_j)-x$ would give both $(F_i\cap G_j)-x$
  and $(F_i\cap G_j)-y$ being pnc-flats of $M$, which is impossible
  since both $x$ and $y$ would need to be isthmuses of $M|F_i\cap G_j$
  for both sets to be flats.  Thus, $F_i\cap G_j$ is a pnc-flat of
  $M$.  The rank assertion follows as above.  This completes the proof
  that $M$ satisfies the properties in Proposition~\ref{prop:lpmchar}
  and so, contrary to the assumption, $M\in \mathcal{L}$.
\end{proof}

\end{document}